\documentclass[a4paper,14pt, reqno]{amsart}
\usepackage{bbm}

\numberwithin{equation}{section}

\usepackage{mathrsfs}
\usepackage{amsmath,amssymb, amsthm}
\usepackage[all,poly,knot]{xy}
\usepackage[colorlinks,linkcolor=black,anchorcolor=black,citecolor=black]{hyperref}
\usepackage{color}

\newtheorem{thm}{Theorem}[section]
\newtheorem{defi}[thm]{{Definition}}
\newtheorem{cor}[thm]{{Corollary}}
\newtheorem{lem}[thm]{{Lemma}}
\newtheorem{prop}[thm]{Proposition}

\newtheorem{exa}[thm]{{Example}}

\newtheorem{note}[thm]{{Notation}}

\def\C{\mathscr C}
\def\Ac{\mathrm{Aut}_{\C}}

\def\Hom{\mathrm{Hom}}
\def\Homc{\Hom_{\C}}
\def\Ho{\Hom_{\C}^0}

\def\Id{\mathrm{Id}}

\def\Mor{\mathrm{Mor}}
\def\Obj{\mathrm{Obj}}

\def\id{\mathrm{id}}
\def\pd{\mathrm{pd}}

\def\lbr{\left(\begin{array}{c}}
\def\lbrt{\left(\begin{array}{cc}}
\def\lbrth{\left(\begin{array}{ccc}}
\def\rbr{\end{array}\right)}

\title[The tensor product of Gorenstein-projective modules]{The tensor product of Gorenstein-projective modules over category algebras}
\author{Ren Wang}
\address{School of Mathematical Sciences, University of Science and Technology of China, Hefei, Anhui 230026, P. R. China}
\email{renw@mail.ustc.edu.cn}
\date{\today}

\begin{document}
\begin{abstract}
For a finite free and projective EI category, we prove that Gorenstein-projective modules over its category algebra are closed under the tensor product if and only if each morphism in the given category is a monomorphism.
\end{abstract}

\maketitle

\section{Introduction}
Let $k$ be a field. Let $\C$ be a finite category, that is, it has
only finitely many morphisms, and consequently it has only finitely
many objects.  Denote by $k$-mod the category of finite
dimensional $k$-vector spaces and $(k\text{-mod})^{\C}$ the category of covariant functors from $\C$ to $k$-mod.

Recall that the category $k\C$-mod of left modules over the category algebra $k\C$ is identified with $(k\text{-mod})^{\C}$; see \cite[Proposition 2.1]{PWebb2}. The category $k\C$-mod is a symmetric monoidal category. More precisely, the tensor product -$\hat{\otimes}$- is defined by \[(M\hat{\otimes} N)(x)=M(x)\otimes_k N(x)\]
for any $M, N\in (k\text{-mod})^{\C}$ and $x\in {\rm Obj}\C$, and $\alpha.(m\otimes n)=\alpha.m\otimes \alpha.n$ for any $\alpha\in \Mor\C, m\in M(x), n\in N(x)$; see~\cite{XF2, XF3}. 

Let $\C$ be a finite EI category. Here, the EI condition means
that all endomorphisms in $\C$ are isomorphisms. In particular,
$\Homc(x,x)=\Ac(x)$ is a finite group for each object $x$. Denote by $k\Ac(x)$ the group algebra.
Recall that
a finite EI category $\C$ is \emph{projective over $k$} if each
$k{\rm Aut}_{\C}(y)$-$k{\rm Aut}_{\C}(x)$-bimodule $k{\rm Hom}_{\C}(x,y)$ is projective on both sides; see~\cite[Definition 4.2]{WR}.

 Denote by $k\C$-Gproj the full subcategory of $k\C$-mod consisting of Gorenstein-projective $k\C$-modules.
We say that $\C$ is \emph{GPT-closed}, if $X,Y\in k\C$-{\rm Gproj} implies  $X\hat{\otimes} Y\in k\C$-{\rm Gproj}.

Let us explain the motivation to study GPT-closed categories. Recall from \cite{WR2} that for a finite projective EI category $\C$, the stable category $k\C$-$\underline{\rm Gproj}$ modulo projective modules has a natural tensor triangulated structure such that it is tensor triangle equivalent to the singularity category of $k\C$. In general, its tensor product is not explicitly given. However, if $\C$ is GPT-closed, then the tensor product -$\hat{\otimes}$- on Gorenstein-projective modules induces the one on $k\C$-$\underline{\rm Gproj}$; see Proposition~\ref{T}. In this case, we have a better understanding of the tensor triangulated category $k\C$-$\underline{\rm Gproj}$.

\begin{prop}\label{NECE}
	Let $\C$ be a finite projective EI category. Assume that $\C$ is GPT-closed. Then each morphism in $\C$ is a monomorphism.
\end{prop}

The concept of a finite \emph{free} EI category is introduced in \cite{LLi}.

\begin{thm}
	Let $\C$ be a finite projective and free EI category. Then the category $\C$ is $GPT$-closed if and only if each morphism in $\C$ is a monomorphism.
\end{thm}

\section{Gorenstein triangular matrix algebras}
In this section, we recall some necessary preliminaries on triangular matrix algebras and Gorenstein-projective modules.

 Recall an $n\times n$ upper
 triangular matrix algebra $\Gamma=\left(
 \begin{array}{cccc}
 R_1 & M_{12} & \cdots & M_{1n} \\
 & R_2 & \cdots & M_{2n} \\
 &  & \ddots & \vdots \\
 &  &  & R_n \\
 \end{array}
 \right)$, where each $R_i$ is an algebra for $1\leq i\leq n$, each $M_{ij}$ is an $R_i$-$R_j$-bimodule for $1\leq i< j\leq n$, and the matrix algebra map is denoted by $\psi_{ilj}: M_{il}\otimes_{R_l} M_{lj} \rightarrow M_{ij}$ for $1\leq i<l<j<t\leq n$; see~\cite{WR}. 
 
 Recall that a left $\Gamma$-module $X=\lbr X_1 \\
 \vdots \\ X_n \\ \rbr$ is described by a column vector, where each $X_i$ is a left $R_i$-module for $1\leq i\leq n$, and the left $\Gamma$-module structure map is denoted by $\varphi_{jl}:
 M_{jl}\otimes_{R_l} X_l \rightarrow X_j$ for $1\leq j<l\leq n$; see~\cite{WR}. Dually, we have the description
 of right $\Gamma$-modules via row vectors. 

\begin{note}\label{B}
	Let $\Gamma_t$ be the algebra given by the $t\times t$ leading
	principal
	submatrix of $\Gamma$ and $\Gamma'_{n-t}$ be the algebra given by cutting the first $t$ rows and the first $t$ columns of $\Gamma$. Denote the left
	$\Gamma_t$-module $\lbr M_{1,t+1}\\
	\vdots \\ M_{t,t+1}\\ \rbr$ by $M_t^*$ and the right
	$\Gamma'_{n-t}$-module $\lbr M_{t,t+1},
	\cdots, M_{tn} \rbr$ by $M_t^{**}$, for $1\leq t\leq n-1$. Denote by $\Gamma_t^D$= {\rm diag}$\lbr R_1,\cdots R_t \rbr$ the sub-algebra of
	$\Gamma_t$ consisting of diagonal matrices, and $\Gamma'_{D,n-t}$= {\rm diag}$\lbr R_{t+1},\cdots R_n \rbr$ the sub-algebra of
	$\Gamma'_{n-t}$ consisting of diagonal matrices.
\end{note}

Let $\Gamma=\lbrt R_1& M_{12}\\ 0 & R_2 \rbr$ be an upper triangular matrix algebra. Recall that the $R_1$-$R_2$-bimodule $M_{12}$ is \emph{compatible}, if the following two conditions hold; see\cite[Definition 1.1]{Z}:
\begin{enumerate}
	\item [(C1)] If $Q^{\bullet}$ is an exact sequence of projective $R_2$-modules, then $M_{12}\otimes_{R_2}Q^{\bullet}$ is exact;
	\item [(C2)] If $P^{\bullet}$ is a complete $R_1$-projective resolution, then $\Hom_{R_1}(P^{\bullet},M_{12})$ is exact.
\end{enumerate}

We use $\pd$ and $\id$ to denote the projective dimension and the injective dimension of a module, respectively.

\begin{lem}\label{C}
Let $\Gamma$ be an upper triangular
matrix algebra satisfying all $R_i$ Gorenstein. If $\Gamma$ is Gorenstein, then each $\Gamma_t$-$R_{t+1}$-bimodule $M^*_t$ is compatible and each $R_t$-$\Gamma'_{n-t}$-bimodule $M^{**}_t$ is compatible for $1 \leq t\leq n-1$.
\end{lem}
\begin{proof}
Let $\Lambda=\lbrt S_1& N_{12}\\ 0 & S_2 \rbr$ be an upper triangular
matrix algebra. Recall the fact that if $\pd_{S_1}(N_{12})<\infty$ and $\pd(N_{12})_{S_2}<\infty$, then $N_{12}$ is compatible; see~\cite[Proposition 1.3]{Z}. Recall that $\Gamma$ is Gorenstein if and only if all bimodules $M_{ij}$ are finitely generated and have finite projective dimension on both sides; see~\cite[Proposition 3.4]{WR}. Then we have $\pd_{R_t}(M^{**}_t)<\infty$ and  $\pd(M^*_t)_{R_{t+1}}<\infty$ for $1 \leq t\leq n-1$. By \cite[Lemma 3.1]{WR}, we have $\pd(M^{**}_t)_{\Gamma'_{n-t}}<\infty$ and  $\pd_{\Gamma_t}(M^*_t)<\infty$ for $1 \leq t\leq n-1$. Then we are done.
\end{proof}
 
Let $A$ be a finite dimensional algebra over a field $k$. Denote by $A$-mod the category of finite dimensional left $A$-modules. The opposite algebra of $A$ is denoted by $A^{\rm op}$. We identify right $A$-modules with left $A^{\rm op}$-modules.

Denote by $(-)^*$ the contravariant functor ${\rm Hom}_A(-,A)$ or ${\rm Hom}_{A^{\rm op}}(-,A)$. Let $X$ be a left $A$-module. Then $X^*$ is a right $A$-module and $X^{**}$ is a left $A$-module. There is an evaluation map ${\rm ev}_X: X\rightarrow X^{**}$ given by ${\rm ev}_X(x)(f)=f(x)$ for $x\in X$ and $f\in X^*$. Recall that an $A$-module $G$ is \emph{Gorenstein-projective} provided that ${\rm Ext}^i_A(G,A)=0={\rm Ext}^i_{A^{\rm op}}(G^*,A)$ for $i\geq 1$ and the evaluation map ${\rm ev}_G$ is bijective; see~\cite[Proposition 3.8]{AB}. 

The algebra $A$ is \emph{Gorenstein} if ${\rm id}_A A<\infty$ and ${\rm id} A_A<\infty$. It is well known that for a Gorenstein algebra $A$ we have ${\rm id}_A A={\rm id} A_A$; see~\cite[Lemma A]{Zaks}. For $m\geq 0$, a Gorenstein algebra $A$ is \emph{$m$-Gorenstein} if ${\rm id}_A A={\rm id} A_A\leq m$.
Denote by $A$-Gproj the full subcategory of $A$-mod consisting of Gorenstein-projective $A$-modules, and $A$-proj the full subcategory of $A$-mod consisting of projective $A$-modules.

The following lemma is well known; see~\cite[Propositions 3.8 and 4.12 and Theorem 3.13]{AB}.

\begin{lem}\label{XL}
	Let $m\geq 0$. Let $A$ be an $m$-Gorenstein algebra. Then we have the following statements.
	\begin{enumerate}
		\item  An $A$-module $M\in A$-{\rm Gproj} if and only if ${\rm Ext}^i_A(M,A)=0$ for all $i>0$.
	    \item  If $M\in A$-{\rm Gproj} and $L$ is a right $A$-module with finite projective dimension, then ${\rm Tor}^A_i(L,M)=0$ for all $i>0$.
	    \item If there is an exact sequence
	    $0\rightarrow M\rightarrow G_0\rightarrow G_1\rightarrow \cdots \rightarrow G_m$
	    with $G_i$ Gorenstein-projective, then $M\in A$-{\rm Gproj}.  \hfill $\square$
	\end{enumerate}
\end{lem}

\begin{lem}\cite[Theorem 1.4]{Z}\label{ZGP}
	Let $M_{12}$ be a compatible $R_1$-$R_2$-bimodule, and $\Gamma=\lbrt R_1& M_{12}\\ 0 & R_2 \rbr$. Then $X=\lbr X_1\\
	X_2\rbr \in \Gamma$-{\rm Gproj} if and only if $\varphi_{12}:
	M_{12}\otimes_{R_2} X_2 \rightarrow X_1$ is an injective $R_1$-map, $\rm{Coker} \varphi_{12}\in R_1$-{\rm Gproj}, and $X_2\in R_2$-{\rm Gproj}.\hfill $\square$
\end{lem}

We have a slight generalization of Lemma~\ref{ZGP}.
\begin{lem}\label{GPE}
	Let $\Gamma$ be a Gorenstein upper triangular
	matrix algebra with each $R_i$ a group algebra. Then $X=\lbr X_1\\ \vdots \\
	X_n\rbr \in \Gamma$-{\rm Gproj} if and only if each $R_t$-map $\varphi_t^{**}=\sum\limits_{j=t+1}^{n}\varphi_{tj}: M_t^{**}\otimes_{\Gamma'_{n-t}} \lbr X_{t+1}\\ \vdots \\
	X_n\rbr \rightarrow X_t$ sending $\lbr m_{t,t+1},
	\cdots, m_{tn} \rbr\otimes \lbr x_{t+1}\\ \vdots \\
	x_n\rbr$ to $\sum\limits_{j=t+1}^{n}\varphi_{tj}( m_{tj}\otimes x_j)$ is injective for $1\leq t\leq n-1$.
\end{lem}
\begin{proof}
	We have that each $R_t$-$\Gamma'_{n-t}$-bimodule $M^{**}_t$ is compatible for $1 \leq t\leq n-1$ by Lemma~\ref{C}.
	
	For the ``only if" part, we use induction on $n$. The case $n=2$ is due to Lemma~\ref{ZGP}. 
	Assume that $n>2$. Write $\Gamma=\lbrt R_1 & M^{**}_1\\ 0 & \Gamma'_{n-1} \rbr$, and $X=\lbr X_1\\
	X'\rbr $. Since $X\in \Gamma$-{\rm Gproj}, by Lemma~\ref{ZGP}, we have that the $R_1$-map $\varphi_1^{**}: M_1^{**}\otimes_{\Gamma'_{n-1}} X' \rightarrow X_1$ is injective and $X'\in \Gamma'_{n-1}$-{\rm Gproj}. By induction, we have that  each $R_t$-map $\varphi_t^{**}: M_t^{**}\otimes_{\Gamma'_{n-t}} \lbr X_{t+1}\\ \vdots \\
	X_n\rbr \rightarrow X_t$ is injective for $1\leq t\leq n-1$.
	
	For the ``if" part, we use induction on $n$. The case $n=2$ is due to Lemma~\ref{ZGP}. Assume that $n>2$.  Write $\Gamma=\lbrt R_1 & M^{**}_1\\ 0 & \Gamma'_{n-1} \rbr$, and $X=\lbr X_1\\
	X'\rbr $. By induction, we have $X'\in \Gamma'_{n-1}$-{\rm Gproj}. Since  the $R_1$-map $\varphi_1^{**}: M_1^{**}\otimes_{\Gamma'_{n-1}} X' \rightarrow X_1$ is injective and its cokernel belongs to $R_1$-{\rm Gproj} for $R_1$ a group algebra, we have $X\in \Gamma$-{\rm Gproj} by Lemma~\ref{ZGP}.
\end{proof}

\begin{cor}\label{GPN}
	Let $\Gamma$ be a Gorenstein upper triangular
	matrix algebra with each $R_i$ a group algebra. Assume that $X=\lbr X_1\\ \vdots \\ X_s \\ 0 \\ \vdots \\ 0\rbr \in \Gamma$-{\rm Gproj}. Then each $R_i$-map $\varphi_{is}: M_{is}\otimes_{R_s} X_s \rightarrow X_i$ is injective for $1\leq i<s\leq n$.	
\end{cor}
\begin{proof}
	We write $X=\lbr X'\\X''\rbr$, where $X''=\lbr X_{i+1}\\ \vdots \\
	X_s\\ \vdots \\0\rbr$ for each $1\leq i<s\leq n$.
    We claim that each $R_i$-map $f_{is} : M_{is}\otimes_{R_s} X_s \rightarrow M_i^{**}\otimes_{\Gamma'_{n-i}} X''$ sending $m_{is}\otimes x_s$ to $\lbr 0, \cdots, m_{is},\cdots,0\rbr\otimes \lbr 0,\cdots,
	x_s, \cdots, 0\rbr^t$ is injective, where $(-)^t$ is the transpose. Since $\varphi_{is}=\varphi_i^{**}\circ f_{is}$ for $1\leq i<s\leq n$, then we are done by Lemma~\ref{GPE}. 
	
	For the claim, we observe that for each $1\leq i<s\leq n$, the $R_i$-map $f_{is}$ is a composition of the following
	\[M_{is}\otimes_{R_s} X_s \overset{g_{is}}{\longrightarrow} \lbr 0,
	\cdots,0,M_{is},\cdots, M_{in} \rbr\otimes_{\Gamma'_{n-i}} X'' \overset{\iota\otimes {\rm Id}}{\longrightarrow} M_i^{**}\otimes_{\Gamma'_{n-i}} X'',\] where the right $\Gamma'_{n-i}$-map $\lbr 0,
	\cdots,M_{is},\cdots, M_{in} \rbr \overset{\iota}{\longrightarrow} M_i^{**}$ is the inclusion map and $g_{is}$ sends $m_{is}\otimes x_s$ to $\lbr 0,
	\cdots, m_{is}, \cdots,0\rbr\otimes \lbr 0,\cdots,
	x_s, \cdots, 0\rbr^t$. We observe a $R_i$-map $\lbr 0,
	\cdots,M_{is},\cdots, M_{in} \rbr\otimes_{\Gamma'_{n-i}} X'' \overset{g'_{is}}{\longrightarrow} M_{is}\otimes_{R_s} X_s$, $\lbr 0,
	\cdots,m_{is},\cdots, m_{in} \rbr\otimes\lbr 0,\cdots,
	x_s, \cdots, 0\rbr^t \mapsto m_{is}\otimes x_s$ satisfying $g'_{is}\circ g_{is}={\Id_{M_{is}\otimes_{R_s} X_s}}$. Hence the $R_i$-map $g_{is}$ is injective.
	We observe that the right $\Gamma'_{n-i}$-modules $\lbr 0,
	\cdots,M_{is},\cdots, M_{in} \rbr$ and $M_i^{**}$ have finite projective dimensions; see~\cite[Lemma 3.1]{WR}, and $X''\in \Gamma'_{n-i}$-{\rm Gproj} by Lemma~\ref{ZGP}. Then the $R_i$-map $\iota\otimes {\rm Id}$ is injective by Lemma~\ref{XL} (2).
\end{proof}

\section{Proof of Proposition 1.1}

Let $k$ be a field. Let $\C$ be a finite category, that is, it has
only finitely many morphisms, and consequently it has only finitely
many objects. Denote by $\Mor\C$ the finite set of all morphisms in
$\C$. The \emph{category algebra} \emph{k}$\C$ of $\C$ is defined as
follows: $\emph{k}\C=\bigoplus\limits_{\alpha \in \Mor\C}k\alpha$ as
a $k$-vector space and the product $*$ is given by the rule
 \[\alpha * \beta=\left\{\begin{array}{ll}
   \alpha\circ\beta, & \text{ if }\text{$\alpha$ and $\beta$ can be composed in $\C$}; \\
    0, & \text{otherwise.}
    \end{array}\right.\]
The unit is given by $1_{k\C}=\sum\limits_{x \in \Obj\C }\Id_x$,
where $\Id_x$ is the identity endomorphism of an object $x$ in $\C$.

If $\C$ and $\mathscr D$ are two equivalent finite categories, then
\emph{k}$\C$ and \emph{k}$\mathscr D$ are Morita equivalent; see
\cite[Proposition 2.2]{PWebb2}. In particular, $k\C$ is Morita
equivalent to $k\C_0$, where $\C_0$ is any skeleton of $\C$. So we
may assume that $\C$ is \emph{skeletal}, that is, for any two
distinct objects $x$ and $y$ in $\C$, $x$ is not isomorphic to $y$.

The category $\C$ is called a \emph{finite EI category} provided
that all endomorphisms in $\C$ are isomorphisms. In particular,
$\Homc(x,x)=\Ac(x)$ is a finite group for any object $x$ in $\C$. Denote by $k\Ac(x)$ the group algebra.

Throughout this paper, we assume that $\C$ is a finite EI category which
is skeletal, and $\Obj\C=\{x_1,x_2,\cdots,x_n\}$, $n\geq 2$, satisfying
$\Homc(x_i,x_j)=\emptyset$ if $i<j$.

Let $M_{ij}=k\Homc(x_j,x_i)$.
Write $R_i=M_{ii}$, which is a group algebra. Recall that the category algebra $k\C$ is isomorphic to the corresponding
upper triangular matrix algebra $\Gamma_{\C}=\left(
          \begin{array}{cccc}
            R_1 & M_{12} & \cdots & M_{1n} \\
             & R_2 & \cdots & M_{2n} \\
             &  & \ddots & \vdots \\
                &  &  & R_n \\
          \end{array}
        \right)$; see~\cite[Section 4]{WR}.

Let us fix a field $k$. Denote by $k$-mod the category of finite
dimensional $k$-vector spaces and $(k\text{-mod})^{\C}$ the category of covariant functors from $\C$ to $k$-mod. 

Recall that the category $k\C$-mod of left modules over the category algebra $k\C$, is identified with $(k\text{-mod})^{\C}$; see \cite[Proposition 2.1]{PWebb2}. The category $k\C$-mod is a symmetric monoidal category. More precisely, the tensor product -$\hat{\otimes}$- is defined by \[(M\hat{\otimes} N)(x)=M(x)\otimes_k N(x)\]
for any $M, N\in (k\text{-mod})^{\C}$ and $x\in {\rm Obj}\C$, and $\alpha.(m\otimes n)=\alpha.m\otimes \alpha.n$ for any $\alpha\in \Mor\C, m\in M(x), n\in N(x)$; see~\cite{XF2, XF3}. 

Let $\C$ be a finite EI category, and $\Gamma=\Gamma_{\C}=\left(
\begin{array}{cccc}
R_1 & M_{12} & \cdots & M_{1n} \\
& R_2 & \cdots & M_{2n} \\
&  & \ddots & \vdots \\
&  &  & R_n \\
\end{array}
\right)$ be the corresponding
upper triangular matrix algebra. Let $X=\lbr X_1\\ \vdots \\ X_n \\ \rbr$ and $Y=\lbr Y_1\\ \vdots \\ Y_n \\ \rbr$ be two $\Gamma$-modules, where the left $\Gamma$-module structure maps are denoted by $\varphi^X_{ij}$ and $\varphi^Y_{ij}$, respectively. We observe that $X\hat{\otimes} Y=\lbr X_1\otimes_k Y_1\\ \vdots \\ X_n\otimes_k Y_n \\ \rbr$, where the module structure map $\varphi_{ij}: M_{ij}\otimes_{R_j}(X_j\otimes_k Y_j)\rightarrow  X_i\otimes_k Y_i$ is induced by the following: $\varphi_{ij}(\alpha_{ij}\otimes (a_j\otimes b_j))=\varphi^X_{ij}(\alpha_{ij}\otimes a_j)\otimes \varphi^Y_{ij}(\alpha_{ij}\otimes b_j)$, where $\alpha_{ij}\in {\rm Hom}_{\C}(x_j,x_i)$, $a_j\in X_j$ and $b_j\in Y_j$.

\begin{defi}
	\emph{Let $\C$ be a finite EI category, and $\Gamma$ be the corresponding upper triangular matrix algebra.
		We say that $\C$ is} GPT-closed, \emph{if $X,Y\in \Gamma$-{\rm Gproj} implies  $X\hat{\otimes} Y\in \Gamma$-{\rm Gproj}.}
\end{defi}

Recall that
a finite EI category $\C$ is \emph{projective over $k$} if each
$k{\rm Aut}_{\C}(y)$-$k{\rm Aut}_{\C}(x)$-bimodule $k{\rm Hom}_{\C}(x,y)$ is projective on both sides; see~\cite[Definition 4.2]{WR}. We recall the fact that the category algebra $k\C$ is Gorenstein if and only if  $\C$ is projective over $k$; see ~\cite[Proposition 5.1]{WR}.

Let $\C$ be a finite projective EI category, and $\Gamma$ be the corresponding upper triangular matrix algebra. Denote by $C_i$ the $i$-th column of $\Gamma$ which is a $\Gamma$-$R_i$-bimodule and projective on both sides.

\begin{prop}\label{ED}
Let $\C$ be a finite projective EI category, and $\Gamma$ be the corresponding upper triangular matrix algebra. Then the following statements are equivalent.
\begin{enumerate}
	\item  The category $\C$ is GPT-closed.
	\item  For any $1\leq p\leq q\leq n$, $C_p\hat{\otimes} C_q\in \Gamma$-{\rm Gproj}. 
	\item  For any $1\leq p\leq q\leq n$, $C_p\hat{\otimes} C_q\in \Gamma$-{\rm proj}. 
\end{enumerate}
\end{prop}
\begin{proof}
``(1)$\Rightarrow$ (2)" and ``(3)$\Rightarrow$ (2)" are obvious.

``(2)$\Rightarrow$ (3)" We only need to prove that the $\Gamma$-module $C_p\hat{\otimes} C_q$ has finite projective dimension, since a Gorenstein-projective module with finite projective dimension is projective. We have $C_p\hat{\otimes} C_q=\lbr M_{1p}\otimes_k M_{1q}\\ \vdots \\ M_{p-1,p}\otimes_k M_{p-1,q}\\R_p\otimes_k M_{pq} \\ 0\\ \vdots\\0\rbr$. Since $\C$ is projective, we have that each $M_{ip}$ is a projective $R_i$-module for $1\leq i\leq p$. Then each $M_{ip}\otimes_k M_{iq}$ is a projective $R_i$-module since $R_i$ is a group algebra for $1\leq i\leq p$. Hence the $\Gamma$-module $C_p\hat{\otimes} C_q$ has finite projective dimension by~\cite[Corollary 3.6]{WR}. Then we are done.

``(2)$\Rightarrow$ (1)" We have that $\Gamma$ is a Gorenstein algebra by~\cite[Proposition 5.1]{WR}. Then there is $d\geq 0$ such that $\Gamma$ is a $d$-Gorenstein algebra. 

For any $M\in \Gamma$-{\rm Gproj}, consider the following exact sequence
\[0\rightarrow M\rightarrow P_0\rightarrow P_1\rightarrow \cdots \rightarrow P_d\rightarrow Y\rightarrow 0\]
with $P_i$ projective, $0\leq i\leq d$. Applying -$\hat{\otimes} N$ on the above exact sequence, we have an exact sequence
\begin{align}\label{ES}
	0\rightarrow M\hat{\otimes} N\rightarrow P_0\hat{\otimes} N\rightarrow P_1\hat{\otimes} N\rightarrow \cdots \rightarrow P_d\hat{\otimes} N\rightarrow Y\hat{\otimes} N\rightarrow 0,
\end{align}
 since the tensor product -$\hat{\otimes}$- is exact in both variables. If $N$ is projective, we have that each $P_i\hat{\otimes} N$ is Gorenstein-projective for $0\leq i\leq d$ by (2). Then we have $M\hat{\otimes} N\in \Gamma$-{\rm Gproj} by Lemma \ref{XL} (3). If $N$ is Gorenstein-projective, we have that each $P_i\hat{\otimes} N$ is Gorenstein-projective for $0\leq i\leq d$ in exact sequence (\ref{ES}) by the above process. Then we have $M\hat{\otimes} N\in \Gamma$-{\rm Gproj} by Lemma \ref{XL} (3). Then we are done.
\end{proof}

The argument in ``(2)$\Rightarrow$ (3)" of Proposition~\ref{ED} implies the following result. It follows that the tensor product -$\hat{\otimes}$- on $\Gamma$-{\rm Gproj} \emph{induces} the one on $\Gamma$-$\underline{\rm Gproj}$, still denoted by -$\hat{\otimes}$-.
\begin{lem}\label{ED1}
 Assume that $\C$ is GPT-closed. Let $M\in \Gamma$-{\rm Gproj} and $P\in \Gamma$-{\rm proj}. Then $M\hat{\otimes} P\in \Gamma$-{\rm proj}. \hfill $\square$
\end{lem}

Let $\C$ be a finite projective EI category, and $\Gamma$ be the corresponding upper triangular matrix algebra of $\C$. Recall that a complex in ${\rm D}^b(\Gamma\text{-mod})$, the bounded derived category of finitely generated left $\Gamma$-modules, is called a \emph{perfect complex} if it is isomorphic to a bounded complex of finitely generated projective modules. Recall from \cite{Buch} that the \emph{singularity category} of $\Gamma$, denoted by ${\rm D}_{sg}(\Gamma)$, is the Verdier quotient category ${\rm D}^b(\Gamma\text{-mod})/{\rm perf}(\Gamma)$, where ${\rm perf}(\Gamma)$ is a thick subcategory of ${\rm D}^b(\Gamma\text{-mod})$ consisting of all perfect complexes. 

Recall from \cite{WR2} that there is a triangle equivalence
\begin{align}\label{ETTF}
F: \Gamma\text{-}\underline{\rm Gproj}\overset{\sim}{\longrightarrow} {\rm D}_{sg}(\Gamma)
\end{align}
sending a Gorenstein-projective module to the corresponding stalk complex concentrated on degree zero. The functor $F$ \emph{transports} the tensor product on ${\rm D}_{sg}(\Gamma)$ to $\Gamma\text{-}\underline{\rm Gproj}$ such that the category $\Gamma\text{-}\underline{\rm Gproj}$ becomes a tensor triangulated category. 

\begin{prop}\label{T}
Let $\C$ be a finite projective EI category, and $\Gamma$ be the corresponding upper triangular matrix algebra of $\C$. Assume that the category $\C$ is GPT-closed. Then the tensor product -$\hat{\otimes}$- on $\Gamma$-$\underline{\rm Gproj}$ induced by the tensor product on $\Gamma$-{\rm Gproj} coincide with the one transported from ${\rm D}_{sg}(\Gamma)$, up to natural isomorphism.
\end{prop}

\begin{proof}
Consider the functor $F$ in (\ref{ETTF}). Recall that the tensor product on ${\rm D}_{sg}(\Gamma)$ is induced by the tensor product -$\hat{\otimes}$- on ${\rm D}^b(\Gamma\text{-mod})$, where the later is given by -$\hat{\otimes}$- on $\Gamma$-mod. We have $F(M)\hat{\otimes} F(N)=F(M\hat{\otimes} N)$ in ${\rm D}_{sg}(\Gamma)$ for any $M,N\in \Gamma$-$\underline{\rm Gproj}$. This implies that $F$ is a tensor triangle equivalence. Then we are done.
\end{proof}
Let $k$ be a field and $G$ be a finite group. Recall that a left (resp. right) $G$-set is a set with a left (resp. right) $G$-action. Let $Y$ be a left $G$-set and $X$ be a right $G$-set.  Recall an equivalence relation ``$\sim$" on the product $X\times Y$ as follows: $(x,y)\sim(x',y')$ if and only if there is an element $g\in G$ such that $x=x'g$ and $y=g^{-1}y'$ for $x,x'\in X$ and $y,y'\in Y$. Write the quotient set $X\times Y/\sim$ as $X\times_G Y$.

The following two lemmas are well known.
\begin{lem}\label{G-Prod}
	Let $Y$ be a left $G$-set and $X$ be a right $G$-set.
	Then there is an isomorphism of $k$-vector spaces
	\[\varphi: kX\otimes_{kG} kY \overset{\sim}{\longrightarrow} k(X\times_G Y),  \quad  x\otimes y\mapsto(x,y),\]
	where $x\in X$ and $y\in Y$. \hfill $\square$
\end{lem}

\begin{lem}\label{Prod}
	Let $Y_1$ and $Y_2$ be two left $G$-sets. Then we have an isomorphism of left $kG$-modules 
	\[\varphi: kY_1\otimes_k kY_2 \overset{\sim}{\longrightarrow} k(Y_1\times Y_2), \quad y_1\otimes y_2\mapsto (y_1,y_2),\] 
	where $y_1\in Y_1, y_2\in Y_2$. \hfill $\square$
\end{lem}

\begin{lem}\label{NEC1}
	Let $\C$ be a finite projective EI category, and $\Gamma$ be the corresponding upper triangular matrix algebra. Assume $1\leq p\leq q\leq n$. Then $C_p\hat{\otimes} C_q\in \Gamma$-{\rm proj} implies that each morphism in $\bigsqcup_{y\in {\rm Obj}\C}\Homc(x_p,y)$ is a monomorphism.
\end{lem}

\begin{proof}
	We have $C_p\hat{\otimes} C_q=\lbr M_{1p}\otimes_k M_{1q}\\ \vdots \\ M_{p-1,p}\otimes_k M_{p-1,q}\\R_p\otimes_k M_{pq} \\ 0\\ \vdots\\0\rbr$. Then each $R_i$-map \[\varphi_{ip}: M_{ip}\otimes_{R_p} (R_p\otimes_k M_{pq}) \rightarrow M_{ip}\otimes_k M_{iq}\] sending $\alpha\otimes (g\otimes \beta)$ to $\alpha\circ g\otimes \alpha\circ\beta$, where $\alpha\in \Homc(x_p,x_i), g\in\Ac(x_p),\beta\in \Homc(x_q,x_p)$, is injective for $1\leq i<p\leq q\leq n$ by Corollary~\ref{GPN}. We have that the sets $\Homc(x_p,x_i)\times_{\Ac(x_p)}( \Ac(x_p)\times \Homc(x_q,x_p))$ and $\Homc(x_p,x_i)\times \Homc(x_q,x_i)$ are $k$-basis of $ M_{ip}\otimes_{R_p} (R_p\otimes_k M_{pq})$ and $M_{ip}\otimes_k M_{iq}$, respectively by Lemma~\ref{G-Prod} and Lemma~\ref{Prod}. For each $1\leq i<p$, since $\varphi_{ip}$ is injective, we have an injective map \[\varphi: \Homc(x_p,x_i)\times_{\Ac(x_p)}( \Ac(x_p)\times \Homc(x_q,x_p))\rightarrow \Homc(x_p,x_i)\times \Homc(x_q,x_i)\] sending $(\alpha,(g,\beta))$ to $(\alpha\circ g,\alpha\circ\beta)$, for $\alpha\in \Homc(x_p,x_i), g\in\Ac(x_p),\beta\in \Homc(x_q,x_p)$.
	
	For each $1\leq i<p$, and $\alpha\in \Homc(x_p,x_i)$, let $\beta,\beta'\in \Homc(x_q,x_p)$ satisfy $\alpha\circ\beta=\alpha\circ\beta'$. Then we have $(\alpha,\alpha\circ\beta)=(\alpha,\alpha\circ\beta')$, that is, $\varphi(\alpha,({\rm Id}_{x_p},\beta))=\varphi(\alpha,({\rm Id}_{x_p},\beta'))$. Since $\varphi$ is injective, we have $(\alpha,({\rm Id}_{x_p},\beta))=(\alpha,({\rm Id}_{x_p},\beta'))$ in $\Homc(x_p,x_i)\times_{\Ac(x_p)}( \Ac(x_p)\times \Homc(x_q,x_p))$. Hence $\beta=\beta'$. Then we have that $\alpha$ is a monomorphism.
\end{proof}

\begin{prop}\label{NECE1}
	Let $\C$ be a finite projective EI category. Assume that $\C$ is GPT-closed. Then that each morphism in $\C$ is a monomorphism.
\end{prop}

\begin{proof}
It follows from Proposition~\ref{ED} and Lemma~\ref{NEC1}.
\end{proof}

Let $\mathcal{P}$ be a finite poset. We assume that ${\rm Obj}\mathcal{P}=\{x_1,\cdots,x_n\}$ satisfying $x_i\nleq x_j$ if $i<j$, and $\Gamma$ is the corresponding upper triangular matrix algebra. We observe that each entry of $\Gamma$ is $0$ or $k$, and each projective $\Gamma$-module is a direct sum of some $C_i$ for  $1\leq i\leq n$.
 For any $a,b\in {\rm Obj}\mathcal{P}$ satisfying $a\nleq b$ and $b\nleq a$, denote by $L_{a,b}=\{x\in {\rm Obj}\mathcal{P}\mid a<x, b<x\}$.

\begin{exa}\label{TCE}
	{\rm Let $\mathcal{P}$ be a finite poset. Then $\mathcal{P}$ is $GPT$-closed if and only if any two distinct minimal elements in $L_{a,b}$ has no common upper bound for $a,b\in {\rm Obj}\mathcal{P}$ satisfying $a\nleq b$ and $b\nleq a$.

For the ``if" part, assume that any two distinct minimal elements in $L_{a,b}$ has no common upper bound. By Proposition~\ref{ED}, we only need to  prove that $C_t\hat{\otimes} C_n$ is projective for $1\leq t\leq n$, since the general case of $C_t\hat{\otimes} C_j$ can be considered in $\Gamma_{max\{t,j\}}$.

For each $1\leq t\leq n$, if $(C_n)_t=k$, that is, $x_n\leq x_t$, then $(C_t)_i=k$ implies $(C_n)_i=k$ for  $1\leq i\leq t$. Hence we have $C_t\hat{\otimes} C_n\simeq C_t$. Assume that $(C_n)_t=0$, that is, $x_n\nleq x_t$. Let $L'_{x_t,x_n}=\{x_{s_1},\cdots,x_{s_r}\}$ be all distinct minimal elements in $L_{x_t,x_n}$. For each $1\leq i<t$, if $(C_t)_i=k=(C_n)_i$, that is, $x_n\leq x_i, x_t\leq x_i$, then there is a unique $x_{s_l}\in L'_{x_t,x_n}$ satisfying $x_{s_l}\leq x_i$, that is, there is a unique $x_{s_l}\in L'_{x_t,x_n}$ satisfying $(C_{s_l})_i=k$, since any two distinct elements in $L'_{x_t,x_n}$ has no common upper bound. Then we have $C_t\hat{\otimes} C_n\simeq \bigoplus\limits_{l=1}^{r} C_{s_l}$. 

For the ``only if" part, assume that $x_t,x_j\in {\rm Obj}\mathcal{P}$ satisfying $x_t\nleq x_j$ and $x_j\nleq x_t$ and $C_t\hat{\otimes} C_j\simeq \bigoplus\limits_{l=1}^{r} C_{s_l}$. Then each $x_{s_l}\in L_{x_t,x_j}$. Assume that $x_{s_1}$ and $x_{s_2}$ be two distinct minimal elements in $L_{x_t,x_j}$ having a common upper bound $x_i$. Then $(C_t\hat{\otimes} C_j)_i=k$ and $(C_{s_1}\oplus C_{s_2})_i=k\oplus k$, which is a contradiction.} \hfill $\square$
\end{exa} 

\section{Proof of Theorem 1.2}
Let $\C$ be a finite EI category. Recall from ~\cite[Definition
2.3]{LLi} that a morphism $x\overset{\alpha}{\rightarrow} y$ in $\C$
is \emph{unfactorizable} if $\alpha$ is not an isomorphism and
whenever it has a factorization as a composite $x
\overset{\beta}{\rightarrow} z \overset{\gamma}{\rightarrow} y$,
then either $\beta$ or $\gamma$ is an isomorphism. Let
$x\overset{\alpha}{\rightarrow} y$ in $\C$ be an unfactorizable
morphism. Then $h\circ\alpha\circ g$ is also unfactorizable for
every $h \in \Ac(y)$ and every $g \in \Ac(x)$; see \cite[Proposition
2.5]{LLi}. Let $x\overset{\alpha}{\rightarrow} y$ in $\C$ be a
morphism with $x\neq y$. Then it has a decomposition
$x=x_0\overset{\alpha_1}{\rightarrow}
x_1\overset{\alpha_2}{\rightarrow} \cdots
\overset{\alpha_n}{\rightarrow} x_n=y$ with all $\alpha_i$
unfactorizable; see \cite[Proposition 2.6]{LLi}.

Following \cite[Definition 2.7]{LLi}, we say that a finite EI
category $\C$ satisfies the Unique Factorization Property (UFP), if
whenever a non-isomorphism $\alpha$ has two  decompositions into
unfactorizable morphisms:
\[x=x_0\overset{\alpha_1}{\rightarrow}
x_1\overset{\alpha_2}{\rightarrow} \cdots
\overset{\alpha_m}{\rightarrow} x_m=y\] and
\[x=y_0\overset{\beta_1}{\rightarrow}
y_1\overset{\beta_2}{\rightarrow} \cdots
\overset{\beta_n}{\rightarrow} y_n=y,\] then $m=n$, $x_i=y_i$, and
there are $h_i\in \Ac(x_i)$, $1\leq i\leq m-1$, such that the following diagram commutes :
\begin{equation*}
\xymatrix@C=0.5cm{
	x=x_0\ar[r]^{\alpha_1} \ar@{=}[d] & x_1 \ar[r]^{\alpha_2} \ar[d]^{h_1} & x_2 \ar[r]^{\alpha_3} \ar[d]^{h_2} &  \cdots \ar[r]^{\alpha_{m-1}}  & x_{m-1} \ar[r]^{\alpha_m} \ar[d]^{h_{m-1}}& x_m=y \ar@{=}[d] \\
	x=x_0\ar[r]^{\beta_1} & x_1 \ar[r]^{\beta_2} & x_2 \ar[r]^{\beta_3} & \cdots \ar[r]^{\beta_{m-1}}  & x_{m-1} \ar[r]^{\beta_m}& x_m=y
}
\end{equation*}

Let $\C$ be a finite EI category. Following \cite[Section 6]{LLi1},
we say that $\C$ is a finite \emph{free} EI category if it satisfies
the UFP. By ~\cite[Proposition 2.8]{LLi}, this is equivalent to the original definition ~\cite[Definition 2.2]{LLi}.

Let $n\geq 2$. Let $\C$ be a finite projective and free EI category with $\Obj\C=\{x_1,x_2,\cdots,x_n\}$ satisfying
$\Homc(x_i,x_j)=\emptyset$ if $i<j$ and $\Gamma$ be the corresponding upper triangular
matrix algebra of $\C$. Then $\Gamma$ is $1$-Gorenstein; see \cite[Theorem 5.3]{WR}. 

Set $\Ho(x_j,x_i)=\{\alpha\in \Homc(x_j,x_i)\mid \alpha \text{
	is unfactorizable}\}$. Denote by $M^0_{ij}=k\Ho(x_j,x_i)$, which is an $R_i$-$R_j$-sub-bimodule of $M_{ij}$; see~\cite[Notation 4.8]{WR}.  Recall the left
$\Gamma_t$-module $M_t^*$ and the right
$\Gamma'_{n-t}$-module $M_t^{**}$ in Notation~\ref{B}, for $1\leq t\leq n-1$. 
Observe that
$M_t^{**}\simeq (M_{t,t+1}^0,M_{t,t+2}^0,\cdots,M_{tn}^0)
\otimes_{\Gamma'_{D,n-t}}\Gamma_{n-t}'$; compare~\cite[Lemmas 4.10 and 4.11]{WR}, which implies that $M_t^*\simeq \Gamma_t\otimes_{\Gamma_t^D} \lbr M_{1,t+1}^0\\ \vdots \\
M_{t,t+1}^0\rbr $.

Let $X=\lbr X_1\\ \vdots \\
X_n\rbr$ be a left $\Gamma$-module. For each $1\leq t\leq n-1$, we have 
\begin{equation*}
	\begin{split}
		M_t^{**}\otimes_{\Gamma'_{n-t}} \lbr X_{t+1}\\ \vdots \\
		X_n\rbr  & \simeq (M_{t,t+1}^0,M_{t,t+2}^0,\cdots,M_{tn}^0)
		\otimes_{\Gamma'_{D,n-t}}\Gamma_{n-t}'\otimes_{\Gamma'_{n-t}} \lbr X_{t+1}\\ \vdots \\
		X_n\rbr\\& \simeq (M_{t,t+1}^0,M_{t,t+2}^0,\cdots,M_{tn}^0)
		\otimes_{\Gamma'_{D,n-t}} \lbr X_{t+1}\\ \vdots \\
		X_n\rbr\\
		& \simeq \bigoplus_{j=t+1}^n M_{tj}^0\otimes_{R_j}X_j.\\
	\end{split}
\end{equation*}
Recall the $R_t$-map $\varphi_t^{**}$ in Lemma~\ref{GPE}. Here, we observe that 
\[ \varphi_t^{**}: \bigoplus_{j=t+1}^n M_{tj}^0\otimes_{R_j}X_j \rightarrow X_t, \quad \sum_{j=t+1}^n (m_j\otimes x_j)\mapsto  \sum_{j=t+1}^n \varphi_{tj} (m_j\otimes x_j).\]

\begin{lem}\label{M}
	Let $\C$ be a finite projective and free EI category and $\Gamma$ be the corresponding upper triangular
	matrix algebra. Assume $1\leq p\leq q\leq n$. If each morphism in $\bigsqcup_{y\in {\rm Obj}\C}\bigsqcup_{j=1}^p \Hom_\C(x_j,y)$ is a monomorphism, then $C_p\hat{\otimes} C_q\in \Gamma$-{\rm proj}.
\end{lem}

\begin{proof}
	We only need to prove that each $R_t$-map 
	\[\varphi_t^{**}: \bigoplus_{j=t+1}^p M_{tj}^0\otimes_{R_j}(M_{jp}\otimes_k M_{jq}) \rightarrow M_{tp}\otimes_k M_{tq}\] is injective for $1\leq t<p$ by Lemma~\ref{GPE} and Proposition~\ref{ED}.

By Lemmas ~\ref{G-Prod} and \ref{Prod}, we have that the set
$\Hom_\C(x_p,x_t)\times \Hom_\C(x_q,x_t)$ is a $k$-basis of $M_{tp}\otimes_k M_{tq}$, and the set
\[\bigsqcup_{j=t+1}^p \Hom^0_\C(x_j,x_t)\times_{\Ac(x_j)}\left( \Hom_\C(x_p,x_j)\times \Hom_\C(x_q,x_j) \right)=:B\]
is a $k$-basis of  $\bigoplus\limits_{j=t+1}^p M_{tj}^0\otimes_{R_j}(M_{jp}\otimes_k M_{jq})$.

We have the following commutative diagram

\begin{equation*}
	\xymatrix@C=1.5cm{
		B \ar[r]^{\subseteq}
		\ar@{->}[d]^{\varphi_t^{**}\mid_B} &  
		\bigoplus\limits_{j=t+1}^p M_{tj}^0\otimes_{R_j}(M_{jp}\otimes_k M_{jq})  \ar@{->}[d]^{\varphi_t^{**}}\\ 
			\Hom_\C(x_p,x_t)\times \Hom_\C(x_q,x_t) \ar[r]^{\subseteq} & M_{tp}\otimes_k M_{tq}}
\end{equation*}

Observe that $\varphi_t^{**}$ is injective if and only if $\varphi_t^{**}\mid_{B}$ is injective for each $1\leq t<p$.

  Assume that  $\varphi_t^{**}(\alpha,(\beta,\theta))=\varphi_t^{**}(\alpha',(\beta',\theta'))$, where $\alpha\in \Hom^0_\C(x_j,x_t)$, $\beta\in \Hom_\C(x_p,x_j)$, $\theta\in \Hom_\C(x_q,x_j)$ and $\alpha'\in \Hom^0_\C(x_{j'},x_t)$, $\beta'\in \Hom_\C(x_p,x_{j'})$, $\theta'\in \Hom_\C(x_q,x_{j'})$. Then we have $\alpha\beta=\alpha'\beta'$ in $\Hom_\C(x_p,x_t)$ and $\alpha\theta=\alpha'\theta'$ in $\Hom_\C(x_q,x_t)$. Since $\C$ is free and $\alpha,\alpha'$ are unfactorizable, we have that $j=j'$ and there is $g\in \Ac(x_j)$ such that $\alpha=\alpha' g$ and $\beta=g^{-1}\beta'$.
Since $\alpha\theta=\alpha'\theta'=\alpha g^{-1}\theta'$ and $\alpha$ is a monomorphism, we have that $\theta=g^{-1}\theta'$. Then we have that $(\alpha,(\beta,\theta))=(\alpha'g, (g^{-1}\beta',g^{-1}\theta'))=(\alpha',(\beta',\theta'))$, which implies that the map $\varphi_t^{**}\mid_{B}$ is injective.
\end{proof}

\begin{thm}
	Let $\C$ be a finite projective and free EI category. Then the category $\C$ is $GPT$-closed if and only if each morphism in $\C$ is a monomorphism.
\end{thm}
\begin{proof}
The ``if" part is just by Proposition~\ref{ED} and Lemma~\ref{M}. The ``only if" part is just by Proposition~\ref{NECE1}.
\end{proof}

\section*{Acknowledgements}
The author is grateful to her supervisor Professor Xiao-Wu Chen for
his guidance. This work is supported by the National Natural
Science Foundation of China (No.s 11522113 and 11571329).

\end{document}